\documentclass[twoside,a4paper,12pt]{amsart}
\usepackage{amsfonts, amsbsy, amsmath, amssymb, latexsym}
\usepackage{mathrsfs,array,enumitem}
\usepackage[top=30mm,right=30mm,bottom=30mm,left=30mm]{geometry}
\usepackage{stmaryrd}
\usepackage{bm}

\usepackage{bbding}

\usepackage{caption,amsmath,amssymb,amsfonts,amsthm,latexsym,graphicx,multirow,color,hyperref,enumerate}
\usepackage[all]{xy}
\usepackage{xcolor}
\usepackage{mathrsfs}
\usepackage{amsmath}
\usepackage{latexsym}
\usepackage{amssymb}
\usepackage{amsfonts}
\usepackage{longtable}
\usepackage{colortbl}
\usepackage{graphics}
\usepackage{booktabs}

\def\FF{\mathbb F}

\def\Aut{{\sf Aut}} 
\def\Out{{\sf Out}}
\def\Cos{{\sf Cos}}

 \def\soc{{\sf soc}}
\def\AG{{\rm AG}}

\def\HA{{\sf HA}} \def\AS{{\sf AS}} \def\PA{{\sf PA}} \def\TW{{\sf TW}}

\def\S{{\rm S}}

\def\AS{{\rm AS}}
\def\HA{{\rm HA}}

\def\PA{{\rm PA}}
\def\SD{{\rm SD}}
\def\CD{{\rm CD}}
\def\HS{{\rm HS}}
\def\HC{{\rm HC}}
\def\TW{{\rm TW}}

\def\Ga{{\Gamma}}

\def\a{\alpha}
\def\b{\beta}

  \def\o{\omega}

 \def\GL{{\rm GL}}\def\PG{{\rm PG}}

\def\Sp{{\rm Sp}}

\def\Sym{{\rm Sym}}

  \def\PGL{{\rm PGL}}
\def\GL{{\rm GL}} 
\def\AGL{{\rm AGL}}

\def\HS{{\rm HS}}

\def\soc{{\rm soc}}
\def\vs{\vskip0.05in}

\def\calP{{\mathcal P}}
\def\calB{{\mathcal B}}
\def\calD{{\mathcal D}}

\def\calI{{\mathcal I}}
\def\calF{{\mathcal F}}

\def\le{\leqslant}
\def\ge{\geqslant}
\def\leq{\leqslant}
\def\geq{\geqslant}

\newtheorem{theorem}{Theorem}[section]%
\newtheorem{lemma}[theorem]{Lemma}%
\newtheorem{corollary}[theorem]{Corollary}%
\newtheorem{proposition}[theorem]{Proposition}%
\newtheorem{definition}[theorem]{Definition}%

\newtheorem{problem}[theorem]{Problem}%
\newtheorem{example}[theorem]{Example}%

\def\qed{{\hfill$\Box$\smallskip}
	\medbreak}

\begin{document}
	
	\title[Locally primitive block designs]{Locally primitive block designs}
	\thanks{The project was partially supported by the NNSF of China (11931005).}
	
	\author{Jianfu Chen}
	\address{Department of Mathematics \\
		%SUSTech International Center for Mathematics\\
		Southern University of Science and Technology,  Shenzhen 518055,  P.R.China}
	\email{chenjf6@sustech.edu.cn}
	%\email{jmchenjianfu@126.com}
	
	\author{Peice Hua}
	\address{%Department of Mathematics \\
		SUSTech International Center for Mathematics\\
		Southern University of Science and Technology,  Shenzhen 518055,  P.R.China}
	\email{huapc@pku.edu.cn}
	
	\author{Cai Heng Li}
	\address{Department of Mathematics \\
		SUSTech International Center for Mathematics\\
		Southern University of Science and Technology,  Shenzhen 518055,  P.R.China}
	\email{lich@sustech.edu.cn}
	
	\author{Yanni Wu}
	\address{Department of Mathematics \\
		%SUSTech International Center for Mathematics\\
		Southern University of Science and Technology,  Shenzhen 518055,  P.R.China}
	%\email{yanniwu2024@163.com}
	\email{12031209@mail.sustech.edu.cn}
	
	\begin{abstract}
		
		A locally primitive 2-design is a 2-design admitting an automorphism group $G$ with primitive local actions.
		It is proved that $G$ is point-primitive, and either $G$ is an almost simple group, or $G$ acting on the points is an affine group.
		
	\end{abstract}
	
	\maketitle
	
	\date\today
	
	\begin{flushleft}
		
		
	\end{flushleft}

	\section{Introduction}

	A {\em 2-design} is a point-block incidence geometry $\calD=(\calP,\calB,\calI)$, where $\calP$ is the set of points, $\calB$ is the set of blocks, and $\calI$ is the incidence relation, such that each block is incident with a constant number of points, and any two points are incident with a constant number of blocks.
	It is known that the number of blocks incident with a given point is also a constant for a 2-design.

	%We denote by $b$ the number of blocks, and $v$ the number of points. and $r$ be the number of blocks which contain a common point.

	The well-known Fisher's inequailty shows that the number of points is less than or equal to the number of blocks for a 2-design, i.e., $|\calP|\leq|\calB|$.
	If $|\calP|=|\calB|$, then $\calD$ is called \textit{symmetric}.
	A {\em flag} of a design $\mathcal{D}$ is an incident point-block pair.
	We denote by $\mathcal{F}$ the set of all flags of $\calD$.
	For a point $\a$ and a block $\b$, the notation $\a \sim \b$ or $\b \sim \a$ means that $\a$ is incident with $\b$.

	An automorphism of a 2-design $\mathcal{D}=(\calP,\calB,\calI)$ is a permutation on the point set $\mathcal{P}$, which preserves the block set $\calB$ and the incidence relation $\calI$.
	All of the automorphisms form the automorphism group $\Aut(\mathcal{D})$.
	For a subgroup $G\le\Aut(\calD)$, we call that $\calD$ is {\em $G$-point-transitive, $G$-block-transitive}, or {\em $G$-flag-transitive}, if $G$ is transitive on $\calP$, $\calB$, or $\calF$, respectively.
	
	A transitive permutation group is called {\em primitive} if it only preserves the trivial partitions of the set,
	and {\em quasiprimitive} if each of its nontrivial normal subgroups is transitive.
	Then, similarly, one can define {\em $G$-point-(quasi)primitive}, {\em $G$-block-(quasi)primitive} and {\em $G$-flag-(quasi)primitive} designs.\vs
	
	Constructing and classifying $2$-designs admitting a block-transitive automorphism group have been a long-term project, and many important results have been obtained in the literature.
	For instance, flag-transitive finite linear spaces, namely, those 2-designs satisfying that any two points are incident with exactly one block, have been well-characterized by Kantor and then a team of six people; see \cite{KantorProplane,Buekenhout2vk1}.
	Further, a nice reduction is given by Camina and Praeger in \cite{Camina-Praeger} for line-transitive and point-quasiprimitive finite linear spaces.
	They showed that such groups are affine or almost simple.
	Note that we traditionally denote by $\lambda$ the number of blocks incident with two common points, and by $r$ the number of blocks incident with a common point.
	Some flag-transitive 2-designs with special $r$ and $\lambda$ have also been studied.
	For instance,
	see \cite{Alavirlam1, Zies1988} for $\gcd\,(r,\lambda)=1$, \cite{RegueiroReduction, Liang} for $\lambda=2$, \cite{YongliChen} for $\lambda$ prime and \cite{Alavirprime} for $r$ prime.
	Symmetric designs with certain automorphism groups are investigated in \cite{Kantor2tr} and \cite{rank3as, rank3ha}, respectively for automorphism groups being 2-transitive on points and being primitive of rank 3.

	For convenience, we denote a point by $\a$, and a block by $\b$.
	Then let $\calD(\a)$ be the set of blocks incident with $\a$, and let $\calD(\b)$ be the set of points incident with $\b$.
	We observe that a 2-design $\calD$ is $G$-flag-transitive if and only if $\calD$ is {\em $G$-locally-transitive}, namely,
	\begin{itemize}
		\item $G_\a$ is transitive on $\calD(\a)$ for all points $\a\in\calP$, and\vskip0.02in
		\item $G_\b$ is transitive on $\calD(\b)$ for all blocks $\b\in\calB$.
	\end{itemize}
	
	In this paper, we study locally primitive designs, defined below.
	
	\begin{definition}
		{\rm
			Let $\calD=(\calP,\calB,\calI)$ be a 2-design, and let $G\le\Aut(\calD)$.
			Then $\calD$ is said to be {\em $G$-locally primitive} if $G_\a$ is primitive on $\calD(\a)$ for all each point $\a\in\calP$, and $G_\b$ is primitive on $\calD(\b)$ for each block $\b\in\calB$.
		}
	\end{definition}

	The $G$-locally primitive finite linear spaces were studied in \cite{locallyprimitivelinearspaces,flagtransitivelinearspaces}, which shows that $G$ acting on $\calP$ is an affine group or an almost simple group.
	We shall prove the same conclusion for the general 2-designs.

	The key property of locally primitive 2-designs we use to achieve our goal is that such designs are point-primitive (see Lemma~\ref{loc-pri1}).
	This allows us to use the O'Nan-Scott-Praeger Theorem for finite (quasi)primitive permutation groups to analyse automorphism groups of locally primitive designs.
	See \cite{ON_3} for a detailed description of the O'Nan-Scott Theorem for primitive groups.
	Praeger \cite{quasi01,quasi02} generalised the O'Nan-Scott Theorem to quasiprimitive groups and showed that a finite quasiprimitive group is one of eight types (with their abbreviations given in parentheses):
	(1) {\em Almost Simple} (\AS);
	(2) {\em Holomorph Affine} (\HA);
	(3) {\em Holomorph Simple} (\HS);
	(4) {\em Holomorph Compound} (\HC);
	(5) {\em Twisted Wreath product} (\TW);
	(6) {\em Simple Diagonal} (\SD);
	(7) {\em Compound Diagonal} (\CD);
	and
	(8) {\em Product Action type} (\PA).
	More detailed information about these types of groups are given in Section \ref{SectionPre}.

	In order to give a reduction for $G$-locally primitive designs, we need to analyse the quasiprimitive types of $G$ acting on the point set $\calP$ and the block set $\calB$.
	There are eight possible types of $G$ acting on $\calP$ whereas there are nine possible types on $\calB$ (non-quasiprimitive action and eight types of quasiprimitive actions).
	Hence there are many cases we may need to tackle with.

	The incidence graph of an incidence geometry provides a different language for and a different view on the geometry.
	
	\begin{definition}
		{\rm
			Let $(\calP,\calB,\calI)$ be an incidence geometry with flag set $\calF$.
			Then the {\em incidence graph} is the bipartite graph $\Ga=(\calP\cup\calB,\calF)$ with biparts $\calP$ and $\calB$ of vertex set, and with edge set $\calF$.
		}
	\end{definition}

	Denote by $\Ga(\a)$ the neighbors of the point $\a$.
	It is not hard to show that an incidence geometry $\calD=(\calP,\calB,\calI)$ is a 2-design if and only if its incidence graph satisfies that for any two vertices $\a_1,\a_2\in\calP$, the intersection of their neighbors has the same size, i.e.,  $|\Ga(\a_1)\cap\Ga(\a_2)|=\lambda$, a constant.
	Besides, given a bipartite graph $\Ga$, we can get an incidence geometry by viewing the two biparts of $\Gamma$ as the point set and the block set, and defining the incidence relation by the edge set of $\Gamma$.
	Therefore, the study of 2-designs can be interpreted as the study of a type of bipartite graphs.
	Since a 2-design is $G$-flag-transitive if and only if its incidence graph is $G$-edge-transitive, the study of flag-transitive 2-designs can also be interpreted as the study of a type of edge-transitive bipartite graphs, which has been an important research topic in algebraic graph theory.

	Similar to the definition of locally primitive designs, one can define {\em locally primitive graphs}.
	The incidence graph of a $G$-locally primitive design is a $G$-locally primitive bipartite graph.
	A systematic study of locally primitive graphs was carried out in \cite{s-arc}, which provides us a useful tool to analyse the locally primitive designs.
	The authors in \cite{s-arc} gave a reduction for $G$-locally primitive bipartite graphs,
	assuming that $G$ acts faithfully and quasiprimitively on the two parts of the bipartition being the orbits of $G$.
	They showed that if $G$ acts on the two vertex-orbits of type $\{X,Y\}$, then $X=Y\in\{\HA, \AS,\TW, \PA\}$, or $\{X,Y\}=\{\SD,\PA\}$ or $\{\CD,\PA\}$.
	Studying locally primitive 2-designs is not only naturally motivated by the classification of 2-designs, but also by the study of locally primitive graphs.
	Giving a reduction or a further classification of locally primitive 2-designs contributes to the classification of flag-transitive designs and enriches the content of locally primitive graphs.

	With the aid of the preliminary result on locally primitive bipartite graphs, we give a reduction for locally primitive 2-designs in the following theorem, as the main result of this paper.

	\begin{theorem}\label{MainThm}
		Let $\mathcal{D}=(\mathcal{P},\mathcal{B},\calI)$ be a $G$-locally primitive $ 2 $-design.
		Then $G$ is flag-transitive and point-primitive on $\calD$, and either
		\begin{itemize}
			\item[\rm(1)] $G$ is an almost simple group and $G^\mathcal{B}$ is quasiprimitive, or\vs
			\item[\rm(2)] $G$ is an affine group, and either
				\begin{itemize}
					\item[\rm(i)] $G^\mathcal{B}$ is also primitive affine, or
					\item[\rm(ii)] $G^\mathcal{B}$ is non-quasiprimitive and $\calD$ is a subdesign of $\AG_i(d,q)$.
				\end{itemize}

		\end{itemize}
		Examples exist for each of the three cases.
	\end{theorem}
	
	The examples for each case in Theorem \ref{MainThm} are given in Section \ref{exam}.
	
	\begin{corollary}\label{cor:sym-design}
		Let $\mathcal{D}=(\mathcal{P},\mathcal{B},\calI)$ be a $G$-locally primitive symmetric 2-design.
		Then $G$ is almost simple or affine, and acts primitively on both $\calP$ and $\calB$.
	\end{corollary}

	Now, we mention a more general object, {\em $t$-design}, which is defined as a point-block incidence geometry $\calD=(\calP,\calB,\calI)$ such that each block is incident with a constant number of points, and any $t$ distinct points are incident with a constant number of blocks.
	%Let $k,t,\lambda$ be positive integers.
	%Let $\calP$ be a finite set of $v$ \textit{points}, and let $\calB$ consist of some $k$-subsets of $\calP$, called {\em blocks}.
	%An incidence geometry $\calD=(\calP,\calB,\calI)$ is called a \textit{$ t $-$ (v,k,\lambda)$ design}, or briefly a {\em $t$-design}, if any $t$ points of $\calP$ are contained in exactly $\lambda$ blocks in $\calB$, where the incidence relation $\calI$ is defined by inclusion.
	It is well known that a $t$-design is a $(t-1)$-design if $t\ge 3$.
	For locally primitive $t$-designs ($t\geq3$), we have the following reduction result.

	\begin{corollary}\label{cor}
		Let $\mathcal{D}=(\mathcal{P},\mathcal{B},\calI)$ be a $G$-locally primitive $t$-design with $t \geqslant 3$.
		Then either $G$ is an almost simple group, or
		$t=3$ and $\mathcal{D}$ is a subdesign of $\AG_i(d,2)$.
	\end{corollary}

	The possible quasiprimitive types for locally primitive $t$-designs are listed in Table \ref{type}, in which we denote by ``$-$'' the type of non-quasiprimitive action.

	\begin{table}[hbt]
		\newcommand{\tabincell}[2]{\begin{tabular}{@{}#1@{}}#2\end{tabular}}
		\centering\normalsize
		\caption{$G$-locally primitive $ t $-design $\mathcal{D}=(\mathcal{P},\mathcal{B},\calI)$}\label{type}
		\scalebox{0.8}[0.8]{
			\begin{tabular}{cccc}
				\toprule[1.5pt]
				$(G^{\mathcal{P}},G^{\mathcal{B}})$ & \tabincell{c}{Possible \\ Types} & Properties \\
				\midrule[1pt]
				$(X,Y)$ & \tabincell{c}{$ (\HA,\HA) $ \\ $ (\AS,\AS) $} & \tabincell{c}{ symmetric,\, $t=2 $ \\ $t\leqslant 6 $}   \\
				\midrule[1pt]
				$(X,-)$ & $ (\HA,-) $ & \tabincell{c}{ non-symmetric,\, $t=2 \text{\,or\,} 3 $ \\ $\mathcal{D}\subseteq \AG_{i}(d,q)$}\\
				\bottomrule[1.5pt]
		\end{tabular}}
	\end{table}
	
	By the reduction given in Theorem~\ref{MainThm}, the following problem naturally occurs.

	\begin{problem}\label{prob:AS-HA}
		{\rm
			Classify $G$-locally primitive $2$-designs where $G$ is an almost simple group or an affine group acting on both points and blocks.
		}
	\end{problem}

	As a continuation of the present paper, we are now preparing for the classification of locally primitive automorphism groups with socle an alternating group, sporadic simple group or simple group of Lie type.
	The study of affine case is on schedule and some results are obtained.
	Besides, as a subfamily of locally primitive 2-designs, a complete classification of locally 2-transitive 2-designs is also in progress.

	The paper is organized as follows.
	In Section~\ref{flag}, we study the connection of incidence graphs, coset graphs and 2-designs.
	In Section~\ref{SectionPre}, we collect some necessary notation, definitions, and some fundamental results, including O'Nan-Scott-Praeger Theorem.
	Then prove a signification lemma for the proof of the main theorem, namely, Lemma \ref{product}.
	In Section~\ref{exam}, we present some examples of 2-designs which are locally primitive.
	In Section~\ref{Primitivity on points and blocks}, we study the primitivity of points and blocks. We give Lemma~\ref{loc-pri1}, which shows that local primitivity implies point-primitivity so that we can use O'Nan-Scott-Praeger Theorem to analyse the automorphism groups of designs.
	In Section~\ref{A reduction}, we prove that the primitive action type of the automorphism groups on points must be affine, almost simple, or product action.
	Finally, we determine the affine case in Section~\ref{X-}, and exclude the product action case in Section~\ref{sec:PA}. Then, the complete proofs of our main results are given in the last Section~\ref{Proof of the main theorem}.

	\section{Incidence coset graphs}\label{flag}

	As pointed out in Introduction, the incidence graph of a flag-transitive design is bipartite and edge-transitive.
	Bipartite edge-transitive graphs can be expressed as coset graphs.
	In the following we shall discuss its connection with 2-designs.
	For more information about coset graphs, one can see \cite[Section 3.2]{s-arc}.

	Given a group $G$ and two subgroups $L,R\le G$, the {\it coset graph} $\Cos(G,L,R)$ is the graph with vertex set $V=[G:L]\cup[G:R]$ where $[G:L]$ and $[G:R]$ are the bipartite parts.
	The adjacency relations is
	\[Lx\sim Ry \Longleftrightarrow Lx\cap Ry\not=\emptyset,\ \mbox{for all $x,y\in G$.}\]
	Clearly, the vertex $L$ is adjacent to the vertex $R$. Note that $G$ acts on the two bipartite parts by the usual right multiplication.
	If $L\cap R$ is core-free in $G$, then $G\leq\Aut(\Gamma)$.
	Further, $\Gamma$ is $G$-edge-transitive and vertex-intransitive with the bipartite parts
	$[G:L]$, $[G:R]$ the two vertex-orbits.

	For convenience, we usually denote the vertices $L,R$ by $\a,\b$, respectively. So  $G_\a=L$ and $G_\b=R$.

	\begin{lemma}\label{lem:coset-g}
		Let $G$ be a finite group and $L,R$ be subgroups of $G$ such that $L\cap R$ is corefree in $G$, and let $\Ga=\Cos\,(G,L,R)$.
		Set $\a=L\in[G:L]$ and $g\in G$.
		Then $|\Ga(\a)\cap\Ga(\a^g)|={|(RL)\cap (RLg)|\over|R|}$.
	\end{lemma}
	\begin{proof}
		Since $\Gamma$ is $G$-edge-transitive, $G_\a$ is transitive on $\Ga(\a)$ and $G_{\a^g}$ is transitive on $\Ga(\a^g)$.
		Note that $\a$ is adjacent to $\b=R\in[G:R]$.
		We have that $\Gamma(\a)=\b^L$ and $\Gamma(\a^g)=(\b^L)^g$.
		Since we identify the vertices $\a$ and $\b$ with $L$ and $R$, respectively, we have that $\a^g=Lg$, and
		\[\begin{array}{rcl}
			\Gamma(\a)&=&\b^{L}=R^L=\{R y\,|\,y\in RL\},\\
			\Ga(\a^g)&=&(\b^{L})^g=\{R y\mid y\in L\}^g=\{R yg\mid y\in L\}=\{R x\mid x\in RL g\}.
		\end{array}\]
		The intersection $\Ga(\a)\cap\Ga(\a^g)$ has the form
		\[\begin{array}{rcl}
			\Ga(\a)\cap\Ga(\a^g)&=&\{R y\,|\,y\in RL\}\cap\{R y\,|\,y\in RL g\}\\
			&=&\{R y\,|\,y\in (RL)\cap (RLg)\}.
		\end{array}\]
		%For any elements $y_1,y_2\in(RL)\cap (RLg)$, we have that $R y_1=R y_2 \Longleftrightarrow y_1y_2^{-1}\in R$.
		Note that $(RL)\cap (RLg)$ is a union of some right cosets of $R$, and  $|\Ga(\a)\cap\Ga(\a^g)|$ in the number of right cosets of $R$ contained in $(RL)\cap (RLg)$.
		It follows that $|\Ga(\a)\cap\Ga(\a^g)|={|(RL)\cap (RLg)|\over|R|}$, as stated.
	\end{proof}

	Let $\mathcal{D}=(\mathcal{P},\mathcal{B},\calI)$ be a $ G $-flag-transitive 2-design.
	Let $(\a,\b)$ be a flag with $\a\in\calP$, $\b\in\calB$.
	As the incidence graph $\Ga$ of $\calD$ is edge-transitive, vertex-intransitive, with bipartite parts $\calP$ and $\calB$, we have
	$$ \Gamma=\Cos\,(G,G_{\alpha},G_{\b}), $$
	with $\calP$ identified as $[G:G_\a] $, and $\calB$ identified as $[G:G_\b]$.
	Obviously, the graph $\Cos\,(G,G_{\alpha},G_{\b})$ is connected.
	Lemma~\ref{lem:coset-g} has the following improved version for incidence graph of a 2-design, which connects coset graphs with 2-designs.

	\begin{lemma}\label{lambda}
		Let $G$ be a finite group and $L,R$ be subgroups of $G$ such that $L\cap R$ is corefree in $G$, and let $\Ga=\Cos\,(G,L,R)$.
		Then the following statements are equivalent
		\begin{itemize}
			\item[{\rm(i)}] $\Ga$ is the incidence graph of a $G$-flag-transitive $2$-design;\vs
			\item[{\rm(ii)}] $|\Ga(\a_1)\cap\Ga(\a_2)|$ is a constant, for any $\a_1,\a_2\in[G:L]$;\vs
			\item[{\rm(iii)}] ${|RL\cap RLg|\over|R|}$ is a constant, for any element $g\in G\setminus L$.
		\end{itemize}
		Consequently, a 2-design $\calD=(\calP,\calB, \calI)$ satisfies ${|G_\b G_\a\cap G_\b G_\a g|\over|G_\b|}=\lambda$ for any flag $(\a,\b)$ and $g\in G\setminus G_{\a}$.
	\end{lemma}

	\begin{proof}
		Firstly, assume that $\Ga$ is the incidence graph of a $G$-flag-transitive $2$-design $(\calP,\calB)$ such that $\calP=[G:G_\a]$.
		Then any two points $\a_1,\a_2\in\calP$ are incident with exactly $\lambda$ common blocks, and so $|\Ga(\a_1)\cap\Ga(\a_2)|=\lambda$ is a constant.
		Thus part~(i) implies part~(ii).

		Secondly, let $\a=L\in[G:L]$ and $g\in G\setminus L$, and let $\a_1=\a$, $\a_2=\a^g$.
		So $|\Ga(\a)\cap\Ga(\a^g)|$ is a constant.
		Then, by Lemma~\ref{lem:coset-g}, ${|RL\cap RLg|\over|R|}=|\Ga(\a)\cap\Ga(\a^g)|$ is a constant.
		Thus part~(ii) implies part~(iii).

		Finally, assume that ${|RL\cap RLg|\over|R|}=\lambda$ for any element $g\in G\setminus L$.
		By Lemma~\ref{lem:coset-g} and the transitivity of $G$ on $[G:L]$, we have $|\Ga(\a)\cap\Ga(\a')|=\lambda$ for any $\a'\in[G:L]$.
		Define an point-block incidence geometry $\calD=(\calP,\calB)$ with $\calP$ and $\calB$ the two biparts of $\Ga$, i.e., $\calP=[G:L]$ and $\calB=[G:R]$.
		Then, any two points $\a_1,\a_2\in\calP$ are incident with exactly the same number of blocks, that is, $\lambda$ blocks by the transitivity of $G$ on $[G:L]$.
		Thus the structure $\calD$ is a 2-design and $\Ga$ is the incidence graph of a $\calD$.
		The flag-transitivity of $G$ on $\calD$ follows from the edge-transitivity of $G$ on $\Gamma$.
	\end{proof}

	Note that the coset graph $\Cos\,(G,L,R)$ is a complete bipartite graph if and only if $G=LR$.
	Accordingly, a $G$-flag-transitive 2-design is trivial such that each of its blocks is incident with every point if and only if $G=G_\a G_\b$ for some $\a\in\calP$ and $\b\in\calB$.
	The incidence graph of such a trivial locally primitive 2-design is exactly a locally primitive complete bipartite graph.
	Such a classification had already been given in
	\cite{LocallyComplete}.

	We mention that the incidence graphs of $2$-designs have small diameters.
	Note that the distance between vertices $u$, $v$ is the length of shortest paths between $u$ and $v$.
	The diameter of a graph is the largest distance between vertices.
	
	\begin{proposition}\label{lem:diam}
		Let $\Ga$ be the incidence graph of a non-trivial $2$-design $\calD$.
		Then $\Ga$ is of diameter at most $4$.
		If further $\calD$ is symmetric, then $\Ga$ is of diameter $3$.
	\end{proposition}
	\begin{proof}
		Let $\calD=(\calP,\calB,\calI)$.
		Then $\Ga$ is a bipartite graph with biparts $\calP$ and $\calB$.
		Let $d(\ ,\ )$ be the distance function.
		
		Note that any two points are of distance $ 2 $. Let $\a\in\calP$ and $\b\in\calB$ such that $\a\not\sim\b$.
		Let $\a'\in\calP$ such that $\a'\sim\b$, and let $\b'\in\Ga(\a)\cap\Ga(\a')$.
		Then $\a\sim\b'\sim\a'\sim\b$, and so $d(\a,\b)=3$.

		If any two blocks are of distance $ 2 $, then $\Ga$ has diameter equal to 3.
		In particular, symmetric $ \calD $ falls into this case. Otherwise, there exists $\b_1,\b_2\in\calB$ with  $d(\b_1,\b_2)>2$.
		Let $\a_1,\a_2\in\calP$ such that $\a_1\sim\b_1$ and $\a_2\sim\b_2$.
		Then there exist $\b\in\calB$ such that $\b\in\Ga(\a_1)\cap\Ga(\a_2)$, and so
		$\b_1\sim\a_1\sim\b\sim\a_2\sim\b_2$, and $d(\b_1,\b_2)=4$.
		Thus, in this case $\Ga$ is of diameter equal to 4.
	\end{proof}

	\section{Preliminaries}\label{SectionPre}

	In this section, we collect notations, definitions and basic properties about designs and their automorphisms.
	Moreover, we prove an important lemma (Lemma \ref{product}) for the proof of the main results of the current paper.

	Let us recall some basic properties of the parameters of a $2$-design $\calD=(\calP,\calB,\calI)$.
	
	%First, by definition \ref{DefDes}, we have
	%	\[2\le k<v,\ \mbox{and}\ \ b=|\calB|< {v \choose k}.\]
	%Let $r$ be the number of blocks containing a fixed point.
	Traditionally, let $v=|\calP|$, $b=|\calB|$, and denote by $k$ the number of points that a block is incident with.
	A 2-design is also called a $2$-$(v,k,\lambda)$ design.
	Then one can obtain
	
	$$vr=bk,$$
	by counting the number of flags $\mathcal{F}=\{(\alpha,\b)\in\calP\times\calB\,|\ \a \sim \b \}$ in two ways.
	Moreover, Fisher's inequality tells us that
	$$b \geqslant v,\ \mbox{and}\ r\ge k.$$
	For a given point $\alpha_0\in\mathcal{P}$, counting flags of the form $\{(\alpha,\b)\,|\,(\alpha_{0},\b)\in\mathcal{F},\alpha\ne\alpha_{0}\}$ in two ways gives rise to the equality:
	$$\lambda(v-1)=r(k-1).$$
	It follows that $ \lambda<r $, and $\lambda v=rk-r+\lambda<rk$.
	
	Given a $t$-$(v,k,\lambda_t)$ design $\mathcal{D}=(\mathcal{P}, \mathcal{B},\calI)$, it is not hard to show that $ \mathcal{D} $ is also an $ s $-$(v,k,\lambda_s)$ design, where $ s \leqslant t $, and in this case,  $\lambda_s=\lambda_t{{v-s \choose t-s}}/{{k-s \choose t-s}}$.
	
	Note that we say a permutation group $G$ on a set $\Omega$ is $s$-homogeneous if $G$ acts transitively on the $s$-subsets of $\Omega$, and is $s$-transitive if $G$ acts transitively on the ordered $s$-tuples of $\Omega$.
	By \cite[Theorems 1.1, 2.1]{t<=6}, a $G$-flag-transitive $t$-designs satisfies that $t\leqslant6$, and $G$ is $\lfloor \frac{t+1}{2} \rfloor$-homogeneous on the points.

	In the present paper, we denote by $G_\a$ (respectively, $G_\b$) the subgroup of $G$ stabilizing the point $\a$ (respectively, the block $\b$).
	The equality $\lambda(v-1)=r(k-1)$ implies that the automorphism group of a $G$-flag-transitive design has large stabilizers, which means that $|G| \leqslant |G_\a|^3.$

	\begin{lemma}\label{large-stab}
		Let $\mathcal{D}$ be $G$-flag-transitive $2$-$(v,k,\lambda)$ design, with a fixed flag $(\a,\b)$. Then \[|G|< {|G_\a|^3\over|G_{\a \b}|^2}.\]
	\end{lemma}
	
	\begin{proof}
		As $\mathcal{D}$ is $G$-flag-transitive, we have $ v={|G|\over |G_\a|} $, $ b={|G|\over |G_\b|}\geqslant v $, $r={|G_\a|\over |G_{\a \b}|}$ and $ k={|G_\b|\over |G_{\a \b}|} $. Then by $ \lambda v<rk $, we have $\lambda\cdot|G|<{|G_\a|^2|G_\b|\over{|G_{\a \b}|^2}}\leqslant{|G_\a|^3\over|G_{\a \b}|^2}.$
	\end{proof}
	
	The bound given in Lemma~\ref{large-stab} allows us to use the tool called large subgroup described in \cite{large-subgp}, to classify locally primitive 2-designs with automorphism groups being almost simple, which will be done in an upcoming sequel, as mentioned in Introduction.

	The following lemma is well known.
	For completeness we provide a proof below.

	\begin{lemma}\label{faithful}	
		Let $\mathcal{D}=(\mathcal{P},\mathcal{B},\calI)$ be a non-trivial $2$-$(v,k,\lambda)$ design, and let $G \leqslant \Aut(\mathcal{D})$.
		Then $G$ acts faithfully on $\mathcal{B}$.
	\end{lemma}
	
	\begin{proof}
		Suppose that $G$ is unfaithful on $\calB$.
		Then there exists some element $g\in G_{(\calB)}\setminus\{1\}$, namely, $g\ne 1$ fixes all the blocks in $\calB$.
		Since $G$ is faithful on $\mathcal{P}$, it follows that $g$ must move some points, say mapping $ \alpha_{1} $ to $ \alpha_{2} $. However, $g$ fixes all the blocks, so the $ r $ blocks incident with
		$\alpha_{1} $ will simultaneously be incident with $ \alpha_{2} $. It implies that $r\leqslant\lambda$, a contradiction. Thus $G$ is faithful on $\calB$.
	\end{proof}
	
	We focus on locally primitive designs $\calD=(\calP,\calB,\calI)$ in this paper, so we introduce a fundamental result --- \textit{O'Nan-Scott-Praeger Theorem} for primitive/quasiprimitive groups.
	
	Let $G$ be a transitive permutation group on a set $\Omega$.
	A partition $\Omega=\Delta_1\cup\dots\cup\Delta_\ell$ is said to be {\em $G$-invariant} if, for any $g\in G$ and any $\Delta_i$,
	\[\Delta_i^g=\Delta_i,\ \mbox{or}\ \Delta_i^g\cap\Delta_i=\emptyset.\]
	Then $G$ induces a transitive action on the set $\{\Delta_1,\dots,\Delta_\ell\}$.
	%\textcolor{red}{Delete ``imprimitive partition'' and ``imprimitive-block''!!!}
	%The partition $\Omega=\Delta_1\cup\dots\cup\Delta_\ell$ is called an {\it imprimitive partition} and each part $\Delta_i$ is said to be an {\it imprimitive-block}.
	Clearly, there are two {\it trivial} $G$-invariant partition, namely, $|\Delta_i|=1$ or $|\Omega|$.
	If the two trivial partitions are the only $G$-invariant partitions, then $G$ is called {\it primitive}, otherwise,
	$G$ is called {\it imprimitive}.
	It is well known that {\it a transitive group $G\le\Sym(\Omega)$ is primitive if and only if the stabilizer of a point in $G$ is a maximal subgroup of $G$.}
	
	A transitive group $G$ on $\Omega$ is said to be \textit{quasiprimitive}, if every non-identity normal subgroup of $G$ is transitive on $ \Omega $.
	Obviously, a primitive group is quasiprimitive, but a quasiprimitive group is not necessarily primitive.
	A fundamental theorem in permutation group theory is the O'Nan-Scott-Praeger Theorem, which describes the structure and action of a primitive/quasiprimitive group; see \cite{quasi01} and \cite{quasi02}.
	
	Let $G\le\Sym(\Omega)$ be quasiprimitive, and take $ \o\in\Omega. $ Let $ N=\soc(G) $ be the {\em socle} of $G$, which is the product of the minimal normal subgroups.
	It is well known (see for example \cite[Theorem 4.3B]{DixonMortimer}) that either $G$ has a unique minimal normal subgroup, or $G$ has exactly two isomorphic, regular minimal normal subgroups.
	So $\soc(G)$ is a direct product of some isomorphic simple groups, namely,
	\[N=\soc(G)=T^m\,\,or\,\,\mathbb{Z}_p^m\]
	where $T$ is a nonabelian simple group and $m$ is a positive integer.

	According to the structures and the actions of $G$ and $N$, we have the following types.

	(I) $G$ has exactly two minimal normal subgroups:

	\begin{itemize}   		
		\item[(1)] {\em Holomorph Simple} (HS):\ $ N=T\times T $, and $ G\le T \rtimes \Aut(T)$;\vskip0.02in
		
		\item[(2)] {\em Holomorph Compound} (HC):\ $ N=T^{\ell}\times T^{\ell} $, and $G\le T^{\ell}\rtimes\Aut(T^{\ell})=(T^{\ell}\times T^{\ell}).\Out(T^{\ell})$.
	\end{itemize}

	(II) $G$ has exactly one minimal normal subgroup:

	\begin{itemize}
		\item[(3)] {\em Almost Simple} (AS):\ $N=T$, and $G\le\Aut(T)$;\vskip0.02in
		
		\item[(4)] {\em Holomorph Affine} (HA):\ $ N=\mathbb{Z}_{p}^{m} $ is regular on $\Omega$, and $G\le\AGL(m,p)$;\vskip0.02in
		
		\item[(5)] {\em Twisted Wreath product} (TW):\ $ N=T^{m} (m>1)$ is regular on $\Omega$, and $G=N \rtimes G_\o$;\vskip0.02in
		
		\item[(6)] {\em Simple Diagonal} (SD):\ $ N=T^{m}(m>1)$, and $N_\o\cong T$;\vskip0.02in
		
		\item[(7)] {\em Compound Diagonal} (CD):\ $ N=T^{m}(m>1)$, and $N_\o\cong T^{\ell}(\ell>1)$ with $\ell$ dividing $m$;\vskip0.02in
		
		\item[(8)] {\em Product Action type} (PA):\ $ N=T^{m}(m>1)$ is irregular on $\Omega$; in the primitive case, $ N_\o\cong H^{m}$, $H< T$.
	\end{itemize}
	
	The O'Nan-Scott-Praeger Theorem is stated as below.
	
	\begin{theorem}\label{thm:O'Nan-Scott}
		Let $G\le\Sym(\Omega)$ be quasiprimitive.
		Then $G$ is one of the eight types described above.
		Further, if $G$ is of type $\HA$, $\HS$ or $\HC$, then $G$ is primitive.
	\end{theorem}

	In the following we give a lemma that plays a key role in the proof of the main results.
	We believe that it could be important in the study of some other certain flag-transitive designs.
	
	\begin{lemma} \label{product}
		Let $\mathcal{D}$ be a $G$-flag-transitive non-trivial $2$-design, with a flag $(\alpha,\b)$ with $\a\in\calP$, $\b\in\calB$.
		Assume that both $G$ and $G_\alpha$ can be decomposed into direct products as
		$$G=G_1 \times G_2,\ G_{\alpha}=L_1 \times L_2=L,\ \text{\ where\ } L_i< G_i. $$
		Assume further that $K=K_1 \times K_2$ is a subgroup of $G_\b$ with $K_1\leqslant G_1$ and $K_2\leqslant G_2$.
		Then $K$ is intransitive on $\calD(\b)$.
	\end{lemma}
	
	\begin{proof}
		Suppose that $K=K_1 \times K_2\le G_\b$ is transitive on $\calD(\b)$, where $ K_i \leqslant G_i.$
		Then $G_\b=KG_{\b\a}$, and hence $G_\b G_{\alpha}=KG_{\b\a}G_\a=KG_\a=KL$.
		Take $x \in G_1 \setminus L_1$ and $y \in G_2 \setminus L_2$. Then for $(1,y),(x,y)\in G \setminus G_{\alpha}$, by Lemma~\ref{lambda},
		$$\lambda=\frac{|KL  \cap KL (1,y)|}{|G_\b|}=\frac{|KL  \cap KL (x,y)|}{|G_\b|}.$$
		Thus $|KL\cap KL (1,y)|=|KL \cap KL (x,y)|$.
		Now
		\[\begin{array}{rcl}
			KL&=&(K_1 \times K_2)(L_1 \times L_2)=K_1 L_1 \times K_2 L_2,\\
			
			KL (1,y)&=&K_1 L_1 \times K_2 L_2y,\\
			
			KL (x,y)&=&K_1 L_1x \times K_2 L_2y.
		\end{array}\]
		Therefore, we have that
		\[\begin{array}{l}
			KL\cap KL (1,y)=K_1L_1\times(K_2L_2\cap K_2L_2y),\\
			KL\cap KL (x,y)=(K_1L_1\cap(K_1L_1x))\times(K_2L_2\cap K_2L_2y).
		\end{array}\]
		Since $|KL\cap KL (1,y)|=|KL \cap KL (x,y)|$, we have that
		$$|K_1 L_1| \times |K_2 L_2  \cap K_2 L_2 y|=|K_1 L_1 \cap K_1 L_1 x| \times |K_2 L_2 \cap K_2 L_2 y|. $$
		Hence $|K_1 L_1|=|K_1 L_1 \cap K_1 L_1 x|$, and so $x\in K_1L_1$.
		It follows that $G_1=K_1 L_1$.
		
		Replacing $(1,y)$ by $(x,1)$, the argument shows that  $G_2=K_2 L_2$.
		Thus
		\[G=G_1\times G_2=K_1 L_1\times K_2L_2=(K_1\times K_2)(L_1\times L_2)=KL=G_\b G_\a.\]
		Then $|G|=|G_\b G_\a|={|G_\a||G_\b|\over|G_\b\cap G_\a|}$, and so
		\[b={|G|\over|G_\b|}={|G_\a|\over|G_\b\cap G_\a|}=k,\]
		which is a contradiction since $\calD$ is a non-trivial 2-design.
		Therefore, $K$ is not transitive on $\calD(\b)$.
	\end{proof}

	\section{Examples}\label{exam}

	In this section, we present some examples of locally primitive designs.  In fact,  we note that every flag-transitive $2$-$(v,k,\lambda)$ design with  $k$ and $r$ both prime is clearly an example of locally primitive 2-design (some examples of almost simple type can be found in \cite{Alavirprime}).
	%Although $G$ acting on $\calP$ is primitive, the action of $G$ on $\calB$ is not necessarily quasiprimitive.
	
	In the following, examples for each possible type described in Theorem \ref{MainThm} are presented.
	To be specific, in each of these examples, a locally primitive group $G$ is one of the types: (1) $G$ is almost simple; (2) both $G^\calP$ and $G^\calB$ are affine groups; and (3) $G^\calP$ is an affine group but $G^\calB$ is non-quasiprimitive.

	\begin{example}\label{ex:AS-AS}{\rm (Projective space $\PG(d,q)$)
			
			Let $V=\FF_q^{d+1}$ be a vector space of dimension $d+1\ge3$ over field $\FF_q$.
			Let
			\[\begin{array}{l}
				\calP=\{\mbox{1-subspaces of $V$}\},\\
				\calB_i=\{\mbox{$(i+1)$-subspaces of $V$}\},
			\end{array}\]
			where $2\le i+1\le d-1$.
			Let $\calD_i=(\calP,\calB_i,\calI)$ and $\calD_i'=(\calP,\calB_i,\calI')$ be incidence geometries with incidence relations defined as follows, for $(\a,\b)\in(\calP,\calB_i)$,
			\[\begin{array}{l}
				\a\overset{\calI}{\sim}\b\Longleftrightarrow \a\subseteq\b,\\
				\a\overset{\calI'}{\sim}\b\Longleftrightarrow \a\nsubseteq\b.
			\end{array}\]
			%Then the linear group $\GL_{d+1}(q)$ is 2-transitive on $\calP$, and transitive on $\calB_i$ for each admissible $i$.
			%Thus for any two elements $\a_1,\a_2\in\calP$ and any element $\b\in\calB_i$, we have that
			%	\[\{\a_1,\a_2\}\subseteq \b\Longleftrightarrow \{\a_1,\a_2\}^g\subseteq \b^g,\ \mbox{for each $g\in G$}.\]
			%It follows that any 2-elements of $\calP$ are contained in a constant number $\lambda$ of elements in $\calB_i$, and so $\calD_i$ is a 2-design.
			The design $\calD_i'$ is called the \textit{complement} of $\calD_i$.

			The number of points is $|\calP|={q^{d+1}-1\over q-1}$.
			The cardinality $|\calB_i|$ is equal to the number of $(i+1)$-subspaces of $V=\FF_q^{d+1}$, which is denoted by $\left[ {d+1 \atop i+1} \right]_q$ and called the Gaussian coefficient.
			We have
			$$\left[ {d+1 \atop i+1} \right]_q
			=\frac{(q^{d+1}-1)(q^{d+1}-q)\cdots(q^{d+1}-q^{i})}{(q^{i+1}-1)(q^{i+1}-q)\cdots(q^{i+1}-q^{i})}.$$
			Then $\calD_i$ is a $2$-$\left({q^{d+1}-1\over q-1}, {q^{i+1}-1\over q-1}, \left[ {d-1 \atop i-1} \right]_q\right)$ design with $r=\left[ {d \atop i} \right]_q$ ($\lambda=1$ if $i=1$), admitting $G=\PGL_{d+1}(q)\le\Aut(\calD_i)=\Aut(\calD_i')$ as a locally primitive automorphism group
			with $G^\calP$ and $G^\calB$ both almost simple (i.e., $( G^\calP ,G^\calB)$ is of type $ (\AS,\AS) $).
			We notice that the design $\calD_i$ is usually denoted by $\PG_{i}(d,q)$.
			The complement $\calD_i'$ of $\calD_i$ is a 2-$\left({q^{d+1}-1\over q-1},{q^{d+1}-q^{i+1}\over q-1},b-2r+\lambda\right)$ design. But $G$ is not necessarily locally primitive on $\calD_i'$.

			In particular,
			if $(d,q,i)=(2,2,1)$, then $\calD=\PG_1(2,2)$ is the well-known Fano plane, a 2-$(7,3,1)$ design. The automorphism group $G=\Aut(\calD)=\PGL_3(2)$ has a subgroup $H=\mathbb{Z}_7 \rtimes \mathbb{Z}_3$ acting locally primitive, with $G^\calP$ and $G^\calB$ both affine (i.e., $(H^\calP ,H^\calB)$ is of type $ (\HA,\HA) $).
			However, the complement $\calD'$ is a 2-$(7,4,2)$ design, and is not $H$-locally primitive.

			In addition, if $(d,q,i)=(3,2,1)$, then $\calD=\PG_1(3,2)$ is a non-symmetric 2-$(15,3,1)$ design.
			This design admits two locally primitive automorphism groups, which are $\PGL_4(2)$ and the alternating group $A_7$.
			These two groups are both almost simple, acting 2-transitively on $\calP$.
			Besides, $A_7$ is also locally primitive on $\PG_2(3,2)$ (a 2-(15,7,3) symmetric design) by the same 2-transitive action.
			
		}
	\end{example}

	%\begin{example}\label{ex:Fano}
	%	{\rm
		%		Let $q^d=2^3$.
		%		Then $\calD=\PG_1(2,2)$ is the well-known 2-$(7,3,1)$ design, called the Fano plane.
		%		The automorphism group $G=\Aut(\calD)=\PGL_3(2)$ is locally primitive, with $G^\calP$ and $G^\calB$ both almost simple. Moreover, $G$ has a subgroup $H=\l h_1\r{:}\l h_2\r=\mathbb{Z}_7{:}\mathbb{Z}_3$, where
		%		$$h_1=
		%		\left[
		%		\begin{array}{ccc}    	    	
			%			0 & 0 & 1  \\
			%			1 & 0 & 1	 \\
			%			0 & 1 & 0
			%		\end{array}
		%		\right]
		%		\text{ \ and \ }
		%		h_2=
		%		\left[
		%		\begin{array}{ccc}    	    	
			%			1 & 0 & 0  \\
			%			0 & 1 & 1	 \\
			%			0 & 1 & 0
			%		\end{array}
		%		\right].$$
		%		Then $\mathcal{D}$ is also $H$-locally primitive, with $G^\calP$ and $G^\calB$ both affine (i.e., $(H^\calP ,H^\calB)$ is of type $ (\HA,\HA) $).
		%However, the complement $\calD'$ is a 2-$(7,4,2)$ design, and is not $H$-locally primitive.
		%	}
	%\end{example}
	
	%\begin{example}\label{ex: 15,3,1}
	%	{\rm
		%		Let $q^d=2^4$. Then $\calD=\PG_1(3,2)$ is a 2-$(15,3,1)$ design.
		%		This design admits two locally primitive automorphism groups, which are
		%		$\PSL_4(2)$ and the alternating group $A_7$.
		%		These two groups are both almost simple, acting 2-transitively on $\calP$.
		%		Besides, $A_7$ is locally primitive on $\PG_2(3,2)$ (a 2-(15,7,3) symmetric design) by the %same 2-transitive action.
		%	}
	%\end{example}

	\noindent\textbf{Remark.}
	The above shows that the Fano plane admits two locally primitive automorphism groups $\PGL_3(2)$ and $\mathbb{Z}_7 \rtimes \mathbb{Z}_3$ of \textit{different types}. The former one is almost simple, whereas the latter one is an affine group on both $\calP$ and $\calB$.
	Besides, we give an example of an almost simple group with the socle an alternating group $A_7$ acting locally primitively on both a symmetric design and a non-symmetric design.

	\begin{example}\label{AG}{\rm (Affine space $\AG(d,q)$)}

		{\rm
			
			Let $V=\mathbb{F}_q^d$ be an vector space with $q$ a prime power. Let
			\[\begin{array}{l}
				\calP=\{ v\, |\, v \in V \},\\
				\calB_i=\{ U+v \ |\ v \in V \text{ and } U \text{ is any $i$-subspace of $V$}\}.
			\end{array}\]
			Denote by $\AG_i(d,q)$ the incidence geometry $(\calP,\calB_i,\calI)$ with incidence relation defined as inclusion. In this geometry, each block is a coset of some $i$-subspace $U$ of $V$.
			Now, if $i\geq2$, then $\AG_i(d,q)$ is a $2$-$(q^d,q^i,\left[ {d-1 \atop i-1} \right]_q)$ design.
			If $i=1$, then $\lambda=1$.

			Note that $G=\AGL_d(q)$ is primitive on $\calP=V$. Meanwhile, since $\GL(V)$ is transitive on the set of $i$-subspaces, we conclude that $G$ is transitive on $\calB$, the set of all cosets of $i$-subspaces.
			Moreover, the point stabilizer $G_0=\GL_d(q)$ is primitive on the set of all $i$-subspaces of $V$, and for an $i$-subspace $U<V$, we have that
			\[G_U^U\cong \widehat{U} \rtimes \GL(U)\cong\AGL(U),\]
			which is primitive on $U$.
			Hence $G=\AGL_d(q)$ is locally primitive on $\AG_i(d,q)$.
			Note that $\widehat{V}\lhd G$ is intransitive on $\calB$, and
			the $G$-invariant block-partition is just the natural parallelism of $\AG_i(d,q)$.
			Thus, $G^\calP$ is an affine group and $G^\calB$ is non-quasiprimitive (i.e., $(G^\calP ,G^\calB)$ is of type $ (\HA,-) $).

			When $q=2$ and $i>1$, $\AG_i(d,q)$ is even a 3-design since any triple on $\mathcal{P}$ corresponds to a $2$-subspace of $V$ by a translation.

			%Let $G=\AGL_d(q)=\widehat{V} \rtimes \GL_d(q).$	
			%    An element $g\in G$ can be written as
			%   \[g=a\widehat w,\ \mbox{where $\widehat w\in\widehat V$ and $a\in\GL_d(q)$}.\]
			%   The action of $g$ on $V$ is defined as
			% \[a\widehat w:\ v\mapsto v^{a\widehat w}=v^a+w.\]
			%For any two distinct points $w$ and $u$ in $\mathcal{P}$, there exists $\widehat{-u} \in \widehat{V}$ such that $w^{\widehat{-u}}=w-u$ and $u^{\widehat{-u}}=u-u=0$.   	
			%Hence the number of blocks containing $w$ and $u$, denoted by $\lambda_{w,u}$, is equal to the number of blocks containing $w-u$ and $0$. This is exactly the number of $i$-dimensional subspaces containing $w-u$. Hence, if $i\geq2$, then $$\lambda_{w,u}=\left[ {d-1 \atop i-1} \right]_q=\frac{(q^{d-1}-q)(q^{d-1}-q^2)\cdots(q^{d-1}-q^{i-2})}{(q^{i-1}-q)(q^{i-1}-q^2)\cdots(q^{i-1}-q^{i-2})},
			%=\frac{(p^{d-1}-1)(p^{d-1}-p)\cdots(p^{d-1}-p^{i-2})}{(p^{i-1}-1)(p^{i-1}-p)\cdots(p^{i-1}-p^{i-2})}
			%$$
			% which is a constant number, independent of the choice of $w$ and $u$. So $\AG_i(d,q)$ is a $2$-$(q^d,q^i,\left[ {d-1 \atop i-1} \right]_q)$ design.
			%If $i=1$, then clearly $\lambda_{w,u}=1$.
			% When $q=2$ and $i>1$, $\AG_i(d,q)$ is even a 3-design since any triple on $\mathcal{P}$ corresponds to a $2$-subspace of $V$ by a translation.
		}
	\end{example}
	
	In the following, we provide an example in the case Theorem \ref{MainThm} (2), that is, $\calD$ is a proper subdesign of $\AG_2(2m,q)$ and admits a locally primitive automorphism group $G$ with $G^\calB$ non-quasiprimitive.

	\begin{example}
		{\rm (A proper subdesign of $\AG_2(2m,q)$)
			
			Let $V=\mathbb{F}_q^{2m}$ be a vector space associated with a non-degenerate alternating form, $ m>1 $.
			Let $\mathcal{D}=(\mathcal{P},\mathcal{B},\calI)$ be an incidence geometry, where $\mathcal{P}=\{ v\, |\, v \in V \}$ and $$ \mathcal{B}=\{ U+v \ |\ v \in V \text{\ and\ }U \text{ is a non-degenerate $2$-subspace of\ }V\}.$$
			Similar to Example \ref{AG}, for any two points $w$ and $u$ in $\mathcal{P}$, the number of blocks containing both $w$ and $u$, denoted by $\lambda_{w,u}$, is equal to the number of non-degenerate $2$-subspaces containing $w_1=w-u$.	
			Every such subspace is in form $\langle w_1, w_2 \rangle$, where $w_2 \in V \setminus \langle w_1 \rangle^\perp$ has $q^{2m}-q^{2m-1}$ choices.
			Hence $$\lambda_{w,u}=\frac{q^{2m}-q^{2m-1}}{q^2-q}=q^{2m-2},$$ is a constant number, independent of the choice of $w$ and $u$.
			Then $\mathcal{D}=(\mathcal{P},\mathcal{B},\calI)$ is a $2$-$(q^{2m},q^2, q^{2m-2})$ design, which is a proper subdesign of $\AG_2(2m,q)$.
			
			Further, let $G=\widehat{V} \rtimes \Sp_{2m}(q)\leq \Aut(\calD)$ with point stabilizer $G_0=\Sp_{2m}(q)$. For $\b=U< V$ a non-degenerate $2$-subspace, i.e., a block containing point $0$, we have that $G_{0 \b}=(\Sp_{2m}(q))_U$ is a maximal subgroup of $G_{0}=\Sp_{2m}(q)$. Besides,
			$$G_\b^{\calD(\b)}\cong \widehat{U} \rtimes \Sp_{2}(q)$$ is primitive on $\calD(\b).$ Hence $G$ is locally primitive on design $\calD$. Note that $\widehat{V}\lhd G$ is not transitive on $\calB$. Then
			$G^\calP$ is an affine group and $G^\calB$ is non-quasiprimitive (i.e., $(G^\calP ,G^\calB)$ is of type $ (\HA,-) $).\vs
			
			In addition, though $ \AG_2(2m,2) $ is a $ 3 $-design, its subdesign $ \calD $ is not a $ 3 $-design: there exist distinct points $ u,v,w\in\mathcal{P}$ such that $ u\not\perp v $ but $ u\perp w $, and then there exists a unique block (a non-degenerate 2-subspace) containing $ \{0,u,v\} $, but no block will contain $ \{0,u,w\} $.
		}
	\end{example}

	\section{Primitivity on points and blocks}\label{Primitivity on points and blocks}
	
	Let $\calD=(\calP,\calB,\calI)$ be a 2-design, and let $G\le\Aut(\calD)$.
	%Then $G$ is {\it block-local-primitive} on $\calD$ if the stabilizer $G_B$ is primitive on $B$ for each $B\in\calB$.
	
	\subsection{Point primitivity}\
	
	The following key lemma shows that locally primitive 2-designs are flag-transitive and point-primitive.
	
	\begin{lemma}\label{loc-pri1}
		Assume that $G$ is locally primitive on $\calD$.
		Then $\mathcal{D}$ is $G$-flag-transitive and $G$-point-primitive.
	\end{lemma}
	
	\begin{proof}
		Let $\Ga$ be the incidence graph of $\calD$.
		Then $\Ga$ is a connected graph.
		Since $G$ is locally primitive on $\calD$, for each element $\b\in\calB$, the stabilizer $G_\b$ is primitive (and so is transitive) on $\calD(\b)$.
		It implies that all flags of $\calD$ are equivalent under $G$ action, and hence $G$ is flag-transitive on $\calD$.
		%{\color{blue} I couldn't see why $G$ is edge-transitive on $E(\Gamma)$. It seems that $G$ is not necessarily block-tr. In the original proof we assume that $G$ is locally primitive.}
		
		Suppose that $G$ is imprimitive on $\calP$.
		Let $\calP=P_1\cup\dots\cup P_m$ be a non-trivial $G$-invariant partition of $\calP$, namely, each permutation of $G$ on $\calP$ induces a permutation of $G$ on the set $\{P_1,\dots,P_m\}$.
		Pick a member $P\in\{P_1,\dots,P_m\}$.
		For each element $\b\in\calB$, since $G_\b$ is primitive on $\calD(\b)$, we have that
		\[|P\cap \calD(\b)|=0, 1, \ \mbox{or}\ |\calD(\b)|.\]
		Take points $\a_1,\a_2\in P$, and $\a_3\notin P$.
		Let $\b_1,\b_2\in\mathcal{B}$ be such that $\a_1,\a_2\in \calD(\b_1)$, $\a_1,\a_3\in \calD(\b_2)$.
		Then $|P\cap\calD(\b_1)|=|\calD(\b_1)|$, and so $P\cap\calD(\b_1)=\calD(\b_1)$, while $\calD(\b_2)\not\subseteq P$.
		Since $G$ is transitive on $\mathcal{B}$, there exists $g\in G$ which maps $\b_1$ to $\b_2$.
		Then we have that
		\[\calD(\b_2)=\calD(\b_1)^g=(\calD(\b_1)\cap P)^g=\calD(\b_2)\cap P^g.\]
		Thus $\calD(\b_2)\subseteq P^g$, so $P\not=P^g$ and then $P\cap P^g=\emptyset$.
		However, $\a_1\in \calD(\b_1)\cap\calD(\b_2)\subseteq P\cap P^g$, which is a contradiction.
		So $\calP$ does not have a non-trivial $G$-invariant partition, and $G$ is primitive on $\calP$.
	\end{proof}

	\subsection{Block imprimitivity}\
	
	Although $G^\calP$ is primitive, the permutation group $G^\calB$ is not necessarily primitive.
	We next establish Lemma~\ref{imprimitiveonB}, which will be needed in the ensuing arguments.
	
	Recall that, for an element $\b\in\calB$, the set of points incident with $\b$ is denoted by  $\calD(\b)=\{\a\in\calP\mid \a\sim\b\}$ for the notation convenience.
	Now for a subset $\Delta$ of $\calB$, we let
	\[\calD(\Delta)=\bigcup_{\b\in\Delta}\calD(\b)=\{\a\in\calP\mid \a\sim\b\in\Delta\},\]
	consisting of points which are incident with some elements in $\Delta$.
	
	Let $\Delta$ with $1<|\Delta|<|\calB|$ be a block of imprimitivity, namely, for any element $g\in G\le\Aut(\calD)$, either $\Delta^g=\Delta$, or $\Delta^g\cap\Delta=\emptyset$.
	
	%We have to notice that the word `block' here is confusing, which is used for an element of $\calB$ in design theory, and also used for a part of a $G$-invariant partition of $\calB$ in permutation group theory.
	%Thus, to avoid confusion, we call the latter `imprimitive-block'.
	
	\begin{lemma} \label{imprimitiveonB}
		Assume that $G$ is imprimitive on $\mathcal{B}$, and let $\Delta$ be a block of imprimitivity for $G$ acting on $\calB$.
		Then the following statements hold.
		\begin{itemize}
			\item[\rm (1)] $G_{\alpha \b }=G_{\alpha \Delta}$, for any $\b\in\Delta$.\vs
			\item[\rm(2)] Any distinct blocks $\b_1, \b_2\in \Delta$ are disjoint, namely, $\calD(\b_1)\cap\calD(\b_2)=\emptyset$.
			%		\item[\rm(3)] If $\calD(\Delta)\not=\calP$, then $G_\Delta=G_{\calD(\Delta)}$.
			% If $\mathcal{P}(S) \ne \mathcal{P}$, then $G_S=G_{\mathcal{P}(S)}$.
		\end{itemize}
	\end{lemma}
	
	\begin{proof}
		(1). Let $\b\in\Delta$.
		Then, as $\mathcal{D}$ is $G$-locally primitive, $G_{\alpha \b}$ is a maximal subgroup of $G_\alpha$.
		Since $\Delta$ is a block of imprimitivity for $G$ acting on $\calB$, we have that $G_\b \leqslant G_\Delta$.
		Thus $G_{\alpha \b} \leqslant G_{\alpha \Delta} \leqslant G_\alpha.$
		
		Suppose that $G_{\alpha \Delta} = G_\alpha$.
		Then $\calD(\alpha)=\b^{G_\alpha} =\b^{G_{\alpha \Delta}} \subset \Delta$.
		%,  i.e., all blocks containing $\alpha$  lying in $S$.
		Let $g\in G$ be such that $\Delta^g\cap\Delta=\emptyset$.
		Then $\calD(\a^g)=\calD(\a)^g\subset\Delta^g$.
		Let $\b'\in\calB$ be such that $\a\sim\b'\sim\a^g$.
		Then $\b'\in\calD(\a)\cap\calD(\a^g)\subset\Delta\cap\Delta^g=\emptyset$, which is a contradiction.
		Thus $G_{\alpha \Delta} < G_\a$, and so $G_{\a\b}=G_{\a\Delta}$ since $G_{\a\b}$ is a maximal subgroup of $G_\a$.\vs

		(2) Suppose that $\calD(\b_1)\cap\calD(\b_2)\not=\emptyset$, where $\b_1,\b_2\in \Delta$ are distinct.
		Let $\a\in \calD(\b_1)\cap\calD(\b_2)$.
		Since $G_\a$ is transitive on $\calD(\a)$, there exists $g \in G_\a$ such that $\b_1^g=\b_2$.
		Thus $\Delta^g=\Delta$, and $g\in G_{\a\Delta}$.
		By part~(1), we have that $g\in G_{\a\Delta}=G_{\a\b_1}$, and so $\b_2=\b_1^g=\b_1$, which is a contradiction.
		So $\calD(\b_1)\cap\calD(\b_2)=\emptyset$, and part~(2) is proved.
	\end{proof}

	\section{A reduction} \label{A reduction}
	
	In this section, let $\mathcal{D}=(\mathcal{P},\mathcal{B},\calI)$ be a $G$-locally primitive 2-design.
	We are going to determine the primitive permutation group $G^\calP$.

	The first lemma characterizes the case where $G$ has a normal subgroup which is regular on $\calP$.
	
	\begin{lemma}\label{twhshc}
		Assume that $G$ has a normal subgroup $N$ which is regular on $\calP$.
		Then $N$ is elementary abelian, and $G$ is a primitive affine group on $\calP$.
	\end{lemma}
	\begin{proof}
		As $N$ is regular on $\calP$,
		by \cite[Proposition 2.3]{tw}, $N$ is solvable.
		Since $G$ is primitive on $\calP$ by Lemma~\ref{loc-pri1}, we conclude that $N$ is elementary abelian.
		Thus $G^\calP$ is a primitive affine group.
	\end{proof}

	The next lemma treats the case where $ G^\calB $ is not quasiprimitive.
	
	\begin{lemma}\label{lem:reg-subgp}
		Assume that $ G^\calB $ is not quasiprimitive.
		Then $G^\calP$ is an affine group.
	\end{lemma}
	\begin{proof}
		Let $1 \ne N \triangleleft G$, such that $N$ is intransitive on $\calB$.
		Thus $N$ is intransitive on the flag set of $\calD$.
		
		Suppose that $N_\a^{\calD(\a)}\not=1$.
		Then the non-trivial normal subgroup $N_\a^{\calD(\a)}$ of the primitive permutation group $G_\a^{\calD(\a)}$ is transitive.
		Since $N$ is transitive on $\calP$, it follows that $N$ is transitive on the flag set, which is a contradiction.
		
		We thus conclude that $N_\a^{\calD(\a)}=1$, namely, $N_\a$ acts trivially on $\calD(\a)$.
		By the result of Zieschang in~\cite[Proposition 2]{Zies1988}, the normal subgroup $N$ is regular on $\calP$.
		By Lemma \ref{twhshc}, we obtain that $G^\calP$ is of type HA, namely, $G^\calP$ is an affine primitive group.
	\end{proof}

	The final lemma of this section deals with the case where $ G^\calB $ is quasiprimitive.

	\begin{lemma} \label{sdcd}
		If $ G^\calB $ is quasiprimitive, then $G^{\mathcal{P}}$ is of type $\HA$, $\AS$, or $\PA$.
	\end{lemma}
	
	\begin{proof}
		Assume that $ G^\calB $ is quasiprimitive.
		Let $N=\soc(G)=T^m,$ where $T$ is nonabelian simple.
		Suppose $G^{\mathcal{P}}$ of type $ \SD $ or $ \CD $.
		Let $ (\alpha,\b)\in(\calP,\calB) $ be a flag.
		Then $N$ is transitive on $\mathcal{P}$ with point stabilizer $N_{\alpha}\cong T^\ell$, where $1 \leqslant \ell<m$.
		Since $G_\a$ is primitive on $\calD(\a)$ and $N_{\alpha}\lhd G_\alpha$, either $N_\a$ is transitive on $\calD(\a)$, or $N_\a$ fixes $\calD(\a)$ pointwise.
		
		Suppose that $ N_{\alpha} $ is transitive on $\calD(\alpha) $.
		Then $r=|\calD(\a)|$ divides $|N_{\alpha}|=|T|^\ell$, and thus $r$ is coprime to $|T|^{m-l}-1=v-1$.
		Since $\lambda(v-1)=r(k-1)$, we have that $r$ divides $\lambda$, which is a contradiction since $r>\lambda$.
		
		We thus conclude that $ N_{\alpha} $ acts trivially on $\calD(\alpha) $.
		Thus $ N_{\a}\leqslant N_\b$ for each element $\b\in\calD(\a)$.
		Since $G$ is primitive on $\calP$ and quasiprimitive on $\calB$, the normal subgroup $N$ is transitive on both $ \mathcal{P} $ and $ \mathcal{B} $.
		Thus $v=|N:N_\a|\ge|N:N_\b|=b$, so that $ v=b $ as it is known that $v\le b$.
		It then follows that $|N_\a|=|N_\b|$ and so $ N_{\a}= N_\b$ as $N_\a\le N_\b$.
		Since the incidence graph $\Ga$ of $\calD$ is connected, it follows that $N_\a$ fixes each element of $\calP\cup\calB$.
		Since $N$ is faithful on $\calP\cup\calB$ by Lemma~\ref{faithful}, we have that $T\le T^\ell \cong N_\a=1$, which is a contradiction.
		Thus $G^\calP$ is not of type SD or type CD.
		
		By Lemma~\ref{twhshc}, $G^\calP$ is not of type HS, HC, or TW, and thus $G^{\mathcal{P}}$ is of type $\HA$, $\AS$, or $\PA$ by Theorem~\ref{thm:O'Nan-Scott}.
	\end{proof}

	Note that for the case where $G^\calP$ is of type $\AS$, Theorem~\ref{MainThm} is satisfied as $G^\calB$ is quasiprimitive by Lemma \ref{lem:reg-subgp}.
	To sum up, in order to prove Theorem \ref{MainThm}, we need to tackle with the cases: (1) $G^{\calP}$ is of type HA and $G^{\calB}$ is non-quasiprimitive; and (2) $G^{\calP}$ is of type PA.
	These are dealt with in the following Sections \ref{X-} and \ref{sec:PA}, respectively.

	\section{Affine groups on points}\label{X-}
	
	In this section, we assume that $\mathcal{D}=(\mathcal{P},\mathcal{B},\calI)$ is a $G$-locally primitive $ 2 $-design, and further, $G^\calP$ is primitive affine, and $G^\calB$ is non-quasiprimitive.
	
	\begin{lemma} \label{orbit}
		Each $N$-orbit on $\mathcal{B}$ is of size equal to $\frac{v}{k}=\frac{b}{r}$.
	\end{lemma}
	
	\begin{proof}
		Let $\Delta$ be an orbit of $N$ on $\calB$.
		Since $N$ is a normal subgroup of $G$, $\Delta^G$ is a $G$-invariant partition of $\calB$, say $\Delta^G=\{\Delta_1,\Delta_2,\dots,\Delta_\ell\}$.
		In particular, the orbits of $N$ on $\calB$ have equal size.
		
		Since $G=NG_\a$, the stabilizer $G_\a$ is transitive on the orbit set $\{\Delta_1,\Delta_2,\dots,\Delta_\ell\}$.
		By Lemma \ref{imprimitiveonB}\,(2), we have that $|\calD(\a)\cap\Delta_i|=0$ or 1 for $1\le i\le\ell$.
		It follows that $|\calD(\a)\cap\Delta_i|=1$ for each $i$ with $1\le i\le\ell$.
		Thus $\ell=|\calD(\a)|=r$, and $|\Delta_i|={|\calB|\over\ell}={b\over r}={v\over k}$.
	\end{proof}
	
	The next lemma determines the designs $\calD$ in the case where $G^\calP$ is affine.
	
	\begin{lemma}\label{HA-}
		Let $N=\mathbb{Z}_p^d$ where $p$ is a prime and $d$ is a positive integer.
		Then $\mathcal{D}$ is a subdesign of $\AG_i(d,q)$ for some $i$ with $1\le i\le d-1$.
		In particular, $\calD$ is non-symmetric.
	\end{lemma}
	
	\begin{proof}
		Regard $\mathcal{P}$ as a vector space $V=\mathbb{F}_p^d$.
		Then $G \leqslant \AGL(d,p)$, and so
		$$G=N \rtimes G_\a=\widehat{V} \rtimes G_0\le\widehat{V} \rtimes \GL(d,p),$$
		where $G_0\le\GL(d,p)$ is irreducible.
		Let $\b\in\calB$ contain the point $\a=0$, and
		let $\Delta=\b^N=\{\b_1,\b_2,\cdots,\b_s\}$.
		Since $G^\calB$ is imprimitive, the $N$-orbits on $\calB$ form a $G$-invariant partition of $\calB$.
		
		If $N_\b=1$, then $N$ is semi-regular on $\calB$ and each $N$-orbit on $\calB$ has length $|N|=v$, which contradicts to Lemma \ref{orbit}
		
		Thus the stabilizer $N_\b$ is a non-identiy proper subgroup of $N$, and so
		$$N_\b =\widehat{W}\cong \mathbb{Z}_p^i,\ 0< i < d.$$
		Since $\mathcal{D}$ is $G$-locally primitive, $N_\b \triangleleft G_\b$ is transitive on $\calD(\b)$.
		Thus we have that
		$$\b=0^{N_\b}=0^{\widehat{W}}=\{0+w\,|\,w \in W\}=W.$$
		Thus every element of $\calB$ containing the point $0$ is an $ i $-subspace of $ V $. In other word,
		$$ \calD(0)=\b^{G_{0}}=\{W_1,\cdots,W_r\}, $$
		where $ W_1,\cdots,W_r< V $ are $ i $-subspaces.
		Therefore we have that
		$$\calB=\b^G=\b^{G_0 N}=\{W_1,\cdots,W_r\}^N=\bigcup_{j=1}^r\,\{W_j+v\,|\,v \in V\},$$
		that is to say, the block set $\calB$ consists of some of the cosets of $i$-subspaces of $V=\mathbb{F}_p^d$.
		So, up to isomorphism, $\mathcal{D}$ is a subdesign of $\AG_i(d,q)$.
		
		Finally, if $\calD$ is symmetric, then the block size $|\b|=k=r$ divides the number of points $v=|\calP|$.
		Since $\lambda (v-1)=r(k-1)$, we get $r$ dividing $\lambda$, which is a contradiction as $r>\lambda$.
		So $\calD$ is non-symmetric.
	\end{proof}

	\section{Product action on points} \label{sec:PA}
	
	In this section, we assume that $\mathcal{D}=(\mathcal{P},\mathcal{B},\calI)$ is a $G$-locally primitive $ 2 $-design, and further, $G^\calP$ is primitive of type $\PA$, and $G^\calB$ is quasiprimitive.
	
	\begin{lemma}\label{lem:PA-on-P}
		Under the assumption made above,
		$G^\calB$ is of type $\PA$, $\SD$ or $\CD$, and $\soc(G)$ is flag-transitive on $\calD$.
	\end{lemma}
	\begin{proof}
		The former statement of the lemma follows from Theorem~1.2 of \cite{s-arc}.
		
		Let $N=\soc(G)$, and let $ (\alpha,\b $) be a flag.
		Suppose that $N_\a$ fixes all elements of $\calD(\a)$, and $N_\b$ fixes all elements of $\calD(\b)$.
		Then $N_\b\le N_\a$ and $N_\a\le N_\b$, so that $N_\a=N_\b$.
		It then follows that $N_\a$ fixes all elements of $\calP\cup\calB$, and so $N_\a=1$ by Lemma~\ref{faithful}, which is not possible since $G^\calP$ is of PA type.
		Thus either $N_\a^{\calD(\a)}\not=1$, or $N_\b^{\calD(\b)}\not=1$.
		Since both $G_\a^{\calD(\a)}$ and $G_\b^{\calD(\b)}$ are primitive, it follows that either $N_\a^{\calD(\a)}$ is transitive, or $N_\b^{\calD(\b)}$ is transitive.
		Note that $N$ is transitive on $\calP$ and $\calB$, it implies that $N$ is transitive on the flags of $\calD$, completing the proof.
	\end{proof}

	To procedure our proof, we need to introduce the structures and the actions of the primitive group $G^\calP$ and its socle $N$.
	Since $G^\calP$ is a primitive permutation group of type PA by assumption, the set $ \mathcal{P} $ can be written as a product of an underlying set $ \Omega $:
	$$\mathcal{P}=\{(\omega_1,\cdots,\omega_m)\,|\,\omega_i \in\Omega \}\cong \Omega^m,\ m>1,$$
	and $N=\soc(G)=T^m \lhd \ G \leqslant H^m \rtimes \S_m,$ where $H \leqslant \Sym(\Omega)$ is an almost simple group and primitive on $\Omega$ with simple socle $T$.
	Since $N$ is a minimal normal subgroup, $G$ acts transitively by conjugation on the set of simple direct factors of $T^m$.
	Then an element $g\in G \leqslant H^m \rtimes \S_m$ has the form $g=(h_1,\cdots,h_m)\sigma,$ where $(h_1,\cdots,h_m) \in H^m$ and $\sigma \in \S_m$, and $g$ acts on $\mathcal{P}$ by
	$$
	(\omega_1,\cdots,\omega_m)^{(h_1,\cdots,h_m)\sigma}=(\omega_{1^{\sigma^{-1}}}^{h_{1^{\sigma^{-1}}}},\cdots,\omega_{m^{\sigma^{-1}}}^{h_{m^{\sigma^{-1}}}}).
	$$	
	
	Here in order to make things clear, we explicitly write the socle $N$ as an direct product in both \textit{internal} and \textit{external} ways, with their differences shown as
	$$N=\begin{cases}
		T_1 \times \cdots \times T_m,\quad & \text{internal with each\ } T_i\leqslant N,\, T_i\cong T ; \\
		T\times \cdots \times T,\quad & \text{external with $ m $ copies of\ } T\nleqslant N.\\
	\end{cases}$$
	For a point $\alpha=(\omega_1,\cdots,\omega_m) \in \mathcal{P},$ the stabilizer can be written as:
	$$N_{\alpha}=\begin{cases}
		(T_1)_{\alpha} \times \cdots \times (T_m)_{\alpha}\text{\ with each\ }(T_i)_{\alpha}\leqslant N_\alpha,\quad & \text{internally};\\
		T_{\omega_1} \times \cdots \times T_{\omega_m}=\{(t_1,\cdots,t_m)\,|\,t_i\in T_{\omega_i}\},\quad & \text{externally}.\\
	\end{cases}$$\vs

	The following three lemmas exclude all the three possible types of $G^\calB$.
	
	\begin{lemma} \label{PACD}
		The quasiprimitive group $G^\calB$ is not of type $\CD$.
	\end{lemma}
	
	\begin{proof}
		Let $ N=\soc(G)$.
		Then $ \mathcal{D} $ is $ N $-flag-transitive by Lemma \ref{lem:PA-on-P}.
		
		Suppose $G^\calB$ is of type CD.
		Then $G$ is built from groups of type $ \SD $ by taking wreath product. Let $ (\alpha,\b) $ be a flag. Then $ N=T_{1}\times\cdots\times T_{\ell t} $ acts transitively on both $\mathcal{P}$ and $\mathcal{B}$, with stabilizers decomposed as
		$$\begin{array}{l}
			N_\alpha=K_{1}\times\cdots\times K_{\ell t}, \text{\ where\ } K_{i}< T_i,\\
			N_\b=X_{1}\times\cdots\times X_{\ell}, \text{where $X_{i}< T_{(i-1)t+1}\times\cdots\times T_{it}$ and $X_{i}\cong T$}.
		\end{array}$$
		This is not possible by Lemma \ref{product}.
	\end{proof}
	
	\begin{lemma}\label{PASD}
		The quasiprimitive group $G^\mathcal{B}$ is not of type $\SD$.
	\end{lemma}
	
	\begin{proof}
		We write the socle
		$$N=\{(t_1,\cdots,t_m)\,|\,t_i\in T\}=T\times \cdots \times T \cong T^m$$ as a direct product in \textit{external} way, where each element $(t_1,\cdots,t_m)$ in $N$ is an ordered $m$-tuple of $T$.
		
		Since $G\cong G^\calB$ is a quasiprimitive permutation group of type $\SD$, we have that
		$$G\leqslant \{(a_1,\cdots,a_m)s\,|\,a_i \in \Aut(T),\ s \in \S_m,\ T a_i=T a_j,\ \text{for any}\ i,j \}.$$
		Without loss of generality, we may assume that $\b\in\mathcal{B}$ is such that
		$$G_\b\leqslant \{(a,a,\dots,a)s\,|\,a \in \Aut(T),\ s \in \S_m\}.$$
		Fix a flag $(\a,\b)\in(\calP,\calB)$.
		Write $N_\b$ and $N_\a$ externally as
		\[\begin{array}{l}
			N_\b=\{(t,\cdots,t)\,|\,t \in T\} \cong T,\\
			N_\alpha=\{(t_1,\cdots,t_m)\,|\,t_i \in Q_i\}=Q_1 \times \cdots \times Q_m,
		\end{array}\]
		where these $Q_i$ are isomorphic subgroups of $T$.
		Let $Y=Q_1 \cap \cdots \cap Q_m<T$.
		Then
		$$N_{\alpha \b}=N_{\alpha}\cap N_{\b}=\{(y,\cdots,y)\,|\,y \in Y\}\cong Y. $$
		
		By Lemma \ref{lem:PA-on-P}, $ \mathcal{D} $ is $ N $-flag-transitive.
		Choose $ g_1 \in T\setminus Q_1 $. We have that $ g=(g_1,1,\cdots,1)\in N \setminus N_{\alpha}$.
		By Lemma~\ref{lambda}, we have that
		$$\lambda=\frac{|N_\b N_\alpha \cap N_\b N_\alpha g|}{|N_\b|}>0.$$  		
		Hence there exists an element $x\in N_\b N_\alpha \cap N_\b N_\alpha g$, and so
		\[x=u_1v_1=u_2v_2g,\]
		for some elements $u_1,u_2\in N_\b$ and $v_1,v_2\in N_\a$, which we write as
		\[\begin{array}{l}
			u_1=(s_1,s_1,\dots,s_1)\in N_\b,\\
			u_2=(s_2,s_2,\dots,s_2)\in N_\b,\\
			v_1=(t_1,t_2,\cdots,t_m)\in N_\a,\\
			v_2=(t'_1,t'_2,\cdots,t'_m)\in N_\a.
		\end{array}\]
		Now $u_2^{-1}u_1=v_2gv_1^{-1}$, and, letting $t=s_2^{-1}s_1$, we have that
		\[\begin{array}{rcl}
			(t,t,\dots,t)&=&(s_2^{-1}s_1,s_2^{-1}s_1,\dots,s_2^{-1}s_1)\\
			&=&u_2^{-1}u_1\\
			&=&v_2gv_1^{-1}\\
			&=&(t'_1,t'_2,\cdots,t'_m)(g_1,1,\cdots,1)(t_1,t_2,\cdots,t_m)^{-1}\\
			&=&(t_1'g_1t_1^{-1},t_2't_2^{-1},\dots,t_m't_m^{-1}).
		\end{array}\]
		
		Then we will have that
		$$t'_1 g_1 t_1^{-1}=t'_2 t_2^{-1}=\cdots=t'_m t_m^{-1}=t\in Q_2\cap Q_3 \cap \cdots \cap Q_m.$$
		If there exists $i_{0}\ge2$ such that $Q_{i{_{0}}}=Q_1$, then
		$$t\in Q_1\cap Q_2 \cap \cdots \cap Q_m.$$
		In particular, $t=t'_1 g_1 t_1^{-1}\in Q_1$ and we get $g_1\in Q_1$, which is a contradiction to the choice of $g_1$.
		It follows that $Q_1\ne Q_i$ for any $i \ne 1$.
		Similarly, we obtain that $Q_i \ne Q_j$ for any $i \ne j$.\vs
		
		\noindent\textbf{Case 1.} Assume that the flag stabilizer $ N_{\alpha \b}\cong Y\ne 1 $.
		
		Let $X=\{(y_1,\cdots,y_m)\,|\,y_i \in Y\}\cong Y^{m}.$
		Then $X\le N_\a$.
		Since $Y<Q_i$, we have that $X<N_\a$, and
		$$1\not=N_{\alpha \b}<X < N_{\alpha}\lhd G_{\alpha}.$$
		Thus $G_{\a\b}\le XG_{\a\b}\le N_\a G_{\a\b}=G_\a$, and $N_{\a\b}\le X\cap G_{\a\b}\le N_\a\cap G_{\a\b}$.
		The latter implies that $X\cap G_{\a\b}=N_{\a\b}$.
		
		Suppose that $G_{\a\b}=XG_{\a\b}$.
		Then $Y^m\cong X\le G_{\a\b}\cap N_\a=N_{\a\b}\cong Y$, which is a contradiction.
		Thus $G_{\a\b}<XG_{\a\b}$.
		
		Suppose that $XG_{\a\b}=N_\a G_{\a\b}=G_\a$.
		Then
		\[{|N_\a||G_{\a\b}|\over|N_{\a\b}|}={|N_\a||G_{\a\b}|\over|N_\a\cap G_{\a\b}|}=|G_\a|={|X||G_{\a\b}|\over|X\cap G_{\a\b}|}={|X||G_{\a\b}|\over|N_{\a\b}|},\]
		and so $|X|=|N_\a|$, which is a contradiction since $X<N_\a$.
		
		We thus conclude that $G_{\a\b}<XG_{\a\b}<G_\a$, which contradicts the fact that $G_{\a\b}$ is a maximal subgroup of $G_\a$.\vs
		
		\noindent\textbf{Case 2.} Assume that the flag stabilizer $ N_{\alpha \b}\cong Y=1 $.
		
		Consider the primitive permutation group $ \overline{G_\b} $ induced by $G_\b$ on $ \mathcal{D}(\b) $, that is, $$\overline{G_\b}=G_\b^{\mathcal{D}(\b)}=G_\b/G_{(\mathcal{D}(\b))}.$$ As $$N\cap G_{(\mathcal{D}(\b))}=N_{(\mathcal{D}(\b))} \leqslant N_{\alpha \b}=1,$$ the image $ \overline{N_\b}$ is isomorphic to $ N_{\b} $, and then $\overline{G_\b}$ has a regular normal subgroup $ \overline{N_\b}\cong N_{\b}\cong T$. Note that $\HS$ is the only possible primitive type with a regular normal subgroup isomorphic to nonabelian simple $ T $.
		Hence $ \overline{G_\b} $ is of type $ \HS $.
		In addition, groups of type $ \HS $ have two regular minimal normal subgroups, which centralize each other.
		Now, suppose that $L<G_{\b}$, such that $ \overline{L} $ is the other regular minimal normal subgroup of $ \overline{G_\b} $. Then we have $\overline{L} \leqslant C_{\overline{G_\b}}(\overline{N_\b})$.		 		
		Explicitly, for any $\bar{g} \in \overline{L}$, where $g=(a,\cdots,a)s \in L < G_{\b}$, and any $\bar{x}\in\overline{N_\b}$, where $ x=(t,\cdots,t) $, $t \in T$, we have $$g^{-1}(t,\cdots,t)g(t,\cdots,t)^{-1}=(t^a t^{-1},\cdots,t^a t^{-1}) \in G_{(\calD(\b))} \cap N_{\b}=1.$$
		Then $t^a=t$ for any $ t\in T $, so $a \in C_{\Aut(T)}(T)={1}$. Hence $g=s$, where $ s\in S_m $, and $L\leqslant S_m$.
		
		Recall that $\alpha\in \calD(\b)$, and the point stabilizer $N_\alpha=Q_1 \times \cdots \times Q_m$ satisfies that $Q_i \ne Q_j \text{\ for\ } i \ne j$. For $1 \ne s \in L \leqslant S_m$,
		$$N_\alpha^s=(Q_1 \times \cdots \times Q_m)^s=Q_{1^{s^{-1}}} \times \cdots \times Q_{m^{s^{-1}}}\ne N_\alpha.$$
		Thus $L$ acts faithfully on $\calD(\b)$, and $L\cong \overline{L}$.
		Now, on the one hand, since $ T $ is transitive on the conjugacy class $\{{Q_1}^t|t\in T\}$, each element $x=(t,\cdots,t) $ in $N_\b$  maps the first component of $N_\alpha$ from $ Q_1 $ to any other element of $\{{Q_1}^t|t\in T\}$.
		On the other hand, $L\cong\overline{L} \triangleleft \overline{G_\b}  $ acts transitively on the stabilizers of points of $ \calD(\b) $ in $N$ by permuting their $ m $ components. Hence $$\{{Q_1}^t|t\in T\} = \{Q_1,\cdots,Q_m\}.$$
		At last, each element in $ Q_2 \cap \cdots \cap Q_m $ fixes all $Q_2,\cdots,Q_m$ by conjugation, so it also fixes $Q_1$, which is hence forced to be the identity, a contradiction to $Q_2 \cap \cdots \cap Q_m \ne 1$.	
		
	\end{proof}

	\begin{lemma}\label{PAPA}
		The quasiprimitive group $G^\calB$ is not of type $\PA$.
	\end{lemma}
	
	\begin{proof}
		Suppose that $G^\mathcal{B}$ is of type PA, with $ N=\soc(G)=T_{1}\times\cdots\times T_{m}$.
		By Lemma \ref{lem:PA-on-P}, $ \mathcal{D} $ is $ N $-flag-transitive.
		Fix a flag $(\alpha,\b)$.
		
		Since $G^\calP$ is primitive, we may write $\calP=\Omega_1\times\dots\times\Omega_m$, so that a point $\a\in\calP$ has the form $\a=(\o_1,\dots,\o_m)$ with $\o_i\in\Omega_i$, and $N_\alpha=X_{1}\times\cdots\times X_{m}$, where $X_{i}< T_i$.
		
		If $G^\calB$ is primitive, then $N_\b=Y_{1}\times\cdots\times Y_{m}$ with $Y_{i}< T_i$,
		which is a contradiction by Lemma \ref{product}.
		Hence $G^{\mathcal{B}}$ is quasiprimitive and imprimitive.
		
		Let $\Delta\subset\calB$ be a maximal block of imprimitivity for the action $G^\calB$.
		Then $\Delta^G$ is a $G$-invariant partition of $\calB$, and $G$ induces a primitive permutation group on $\Delta^G$ with stabilizer $G_\Delta$.
		Thus $\Delta^G=\Sigma_1\times\dots\times\Sigma_m$, and $N_\Delta=H_1\times\dots\times H_m$ with $H_i<T_i$ such that $H_i$ is transitive on $\Sigma_i$ with $1\le i\le m$.
		
		Without loss of generality, pick an element $\b\in\Delta$ and two points $\a_1,\a_2\in\calD(\b)$ with
		$$\alpha_1=(\omega_1,\omega_2,\cdots,\omega_m), \ \alpha_2=(\omega'_1,\omega_2,\cdots,\omega_m),$$
		where $\omega_1 \ne \omega'_1$.
		Then there exists $x\in N_\b$ such that $\a_1^x=\a_2$. Note that $G_\b\le G_\Delta$.
		Thus $x\in N_\Delta=H_1\times\dots\times H_m$ with $H_i<T_i$.
		Let $x=(t_1,t_2,\dots,t_m)$ with $t_i\in H_i$.
		Thus $y=(1,t_2^{-1},\dots,t_m^{-1})\in N_\Delta$, and
		\[x_0=xy=(t_1,1,\dots,1)\in N_\Delta,\ \mbox{and}\ \a_1^{x_0}=\a_2.\]
		Then $\alpha_2=\alpha_1^{x_0}\in\calD(\b)^{x_0}\cap \calD(\b)$.
		Since $\calD(\b)^{x_0}\cap \calD(\b)\neq\emptyset$, it follows that $\b^{x_0}=\b$ by Lemma~\ref{imprimitiveonB}, and so $x_0\in N_{\b}$.
		Let
		\[R=x_0^{G_\b}=\langle (t_1,1,\cdots,1)^g\,|\,g\in G_{\b}\rangle.\]
		Then $R\lhd G_{\b}$, and $R^{\calD(\b)}\not=1$ since $x_0$ sends $\a_1$ to $\a_2$,
		and hence the normal subgroup $R^{\calD(\b)}\lhd G_\b^{\calD(\b)}$ is transitive.
		
		Furthermore, any element $g\in G_\b$ has the form $g=(g_1,\dots,g_m)\pi$ where $g_i\in T_i$ and $\pi\in\S_m$, and so
		\[x_{0}^{g}=(t_1,1,\cdots,1)^{g}=(t_1^{g_1},\cdots,1^{g_m})^s=(1,\cdots,1,t_{1^{\pi^{-1}}}^{g_1},1,\cdots,1).\]
		It follows that $R=R_1\times\dots\times R_m$ is a direct product of the isomorphic subgroups $R_1,\dots,R_m$, where each $ R_i\le T_i $.
		This contradicts Lemma \ref{product}, and the lemma is proved.
	\end{proof}

	\section{Proof of the main theorem} \label{Proof of the main theorem}
	
	In this final section, we complete the proofs of the main theorem and its corollaries.
	
	{\bf Proof of Theorem~\ref{MainThm}:}\
	
	Let $\calD=(\calP,\calB,\calI)$ be a 2-design, and assume that $G\le\Aut(\calD)$ is locally primitive on $\calD$.
	
	By Lemma~\ref{faithful}, the group $G$ is faithful on both $\calP$ and $\calB$. Further, by Lemma~\ref{loc-pri1}, the group $G$ is transitive on the flag set and primitive on the point set $\calP$.
	Thus $G^\calP$ is a primitive permutation group, and $G^\calB$ is a transitive permutation group.
	
	First, if $G^\calB$ is not quasiprimitive, then $G^\calP$ is a primitive affine group by Lemma~\ref{lem:reg-subgp}, and then Lemma~\ref{HA-} shows that Theorem~\ref{MainThm} is satisfied.
	
	Otherwise, assume that $G^\calB$ is quasiprimitive.
	Then the primitive permutation group $G^\calP$ is of type HA, AS or PA by Lemma~\ref{sdcd}.
	Moreover, if $G^\calP$ is of type AS, then Theorem~\ref{MainThm} is satisfied, and the case where $G^\calP$ is of type PA is excluded by Lemmas \ref{PACD}-\ref{PAPA}.
	
	We thus only need to treat the case where $G^\calP$ is of type HA.
	Then $G$ has a unique minimal normal subgroup, which is abelian, and it then follows that the quasiprimitive permutation group $G^\calB$ is a primitive affine group.
	This completes the proof of Theorem~\ref{MainThm}.
	\qed
	
	{\bf Proof of Corollary~\ref{cor:sym-design}:}
	
	Let $\calD=(\calP,\calB,\calI)$ be a symmetric 2-design, and assume that $G\le\Aut(\calD)$ is locally primitive on $\calD$.
	Then we obtain by Theorem~\ref{MainThm} that $G^\calP$ is primitive, and either almost simple or affine.
	
	Since the dual structure of $\calD$ also forms a $G$-locally primitive symmetric design, again we obtain by Theorem~\ref{MainThm} that $G^\calB$ is primitive, and either almost simple or affine.
	The proof then follows.
	\qed

	{\bf Proof of Corollary~\ref{cor}:}
	
	Let $\calD=(\calP,\calB,\calI)$ be a $t$-design with $t\ge3$. By \cite[Thm.3.3.2]{faithful}, $\calD$ is non-symmetric. Let $G\le\Aut(\calD)$ be locally primitive on $\calD$. To complete the proof, by Theorem~\ref{MainThm} we may assume that $\calD$ is a subdesign of $\AG_i(d,q)$.
	
	Suppose that $ q>2 $. Let $x$ be a non-zero vector of $\calP=\FF_q^d$, and
	then $\{0,x,ax\}$ is a $ 3 $-subset of $\calP$, where $a\in\mathbb{F}_{q}\setminus\{0,1\}$. Given a block  $\b \in\calB$, note that $ \b $ contains $ 0\in\calP $ if and only if $ \b $ is a subspace of $ V $. Then, $ \b $ contains $\{0,x,ax\}$ if and only if it contains the $ 2 $-subset $\{0,x\}$.
	Since $\calD$ is a 2-design and a 3-design, it follows that $\lambda_3=\lambda_2$, which is a contradiction as $\lambda_2=\lambda_3 {v-2 \choose 3-2}/{k-2 \choose 3-2}$.

	We thus conclude that $q=2$.
	Since $\calD$ is non-trivial, we have that $d>2$. We already know that $\AG_i(d,2)$ is a 3-design when $i>1$. Suppose that its subdesign $\calD$ is a $ 4 $-design.  Let $x,y\in\calP=\FF_2^d$ be two linearly independent vectors, and
	then $\{0,x,y,x+y\}$ is a $ 4 $-subset of $\calP$. Similarly, a block contains $ 0\in\calP $ if and only if it is a subspace of $ V $; then, it contains  $\{0,x,y,x+y\}$ if and only if it contains $\{0,x,y\}$. It follows that $\lambda_4=\lambda_3$, a contradiction.\qed

	\subsection*{Declaration of Competing Interests}
	The authors have no known financial interests or other competing interests that are relevant to the content of this paper to declare.

	\subsection*{Data Availability}
	
	Data sharing not applicable to this article as no datasets were generated or analyzed during the current study.

	\subsection*{Acknowledgements}
	
	The project was partially supported by the NNSF of China (11931005).

	\bibliographystyle{siam}
	\bibliography{ref}

\end{document}